\documentclass[11pt, notitlepage]{article}
\usepackage{amssymb,amsmath,comment}
\catcode`\@=11 \@addtoreset{equation}{section}
\def\thesection{\arabic{section}}

\def\theequation{\thesection.\arabic{equation}}
\catcode`\@=12
\usepackage{colortbl}
\usepackage[mathscr]{eucal}
\usepackage{epsf}
\usepackage{esint}
\usepackage{a4wide}
\def\R{\mathbb{R}}

\DeclareMathOperator*{\esssup}{ess\,sup}
\DeclareMathOperator*{\essinf}{ess\,inf}
\DeclareMathOperator*{\essliminf}{ess\,lim\,inf}
\DeclareMathOperator*{\esslimsup}{ess\,lim\,sup}

\newcommand{\noi} {\noindent}

\newcommand{\loc} {\mathrm{loc}}

\usepackage[all]{xy}
\catcode`\@=11
\def\theequation{\@arabic{\c@section}.\@arabic{\c@equation}}
\catcode`\@=12

\newtheorem{Theorem}{Theorem}[section]
\newtheorem{Lemma}[Theorem]{Lemma}

\newtheorem{Corollary}[Theorem]{Corollary}
\newtheorem{Remark}[Theorem]{Remark}
\newtheorem{Definition}[Theorem]{Definition}

\begin{document}

{\vspace{0.01in}}

\title
{Continuity of weak solutions to a class of doubly nonlinear parabolic nonlocal $p$-Laplace equation}

\title
{\sc Lower semicontinuity and pointwise behavior of  supersolutions for some doubly nonlinear nonlocal parabolic $p$-Laplace equations}
\author{Agnid Banerjee, Prashanta Garain and Juha Kinnunen}

\maketitle

\begin{abstract}
{We discuss pointwise behavior of weak supersolutions for a class of  doubly nonlinear parabolic fractional $p$-Laplace equations which includes the fractional parabolic $p$-Laplace equation and the fractional porous medium equation. More precisely, we show that weak supersolutions have lower semicontinuous representative. We also prove that the semicontinuous representative at an instant of time is determined by the values at previous times. 
This gives a pointwise interpretation for a weak supersolution at every point. 
The corresponding results hold true also for weak subsolutions.
Our  results extend  some recent results in the local parabolic case  and in the nonlocal elliptic case   to the nonlocal parabolic case. 
We prove the required energy estimates and measure theoretic De Giorgi type lemmas in the fractional setting.}
\medskip

\noi {Key words: Doubly nonlinear parabolic equation, fractional $p$-Laplace equation, porous medium equation, energy estimates, De Giorgi's method}.

\medskip

\noi{\textit{2010 Mathematics Subject Classification:} 35K59, 35K65, 35B45, 35B65, 35R11.}
\end{abstract}

\section{Introduction}
This article discusses pointwise behavior of weak supersolutions $u:\mathbb{R}^n\times(0,T)\to\mathbb{R}$ to the doubly nonlinear parabolic nonlocal $p$-Laplace equation
\begin{equation}\label{maineqn}
\partial_t(|u|^{q-1}u)+\mathcal{L}u=0\text{ in }\Omega\times(0,T),
\quad 1<p<\infty,\quad q>0,
\end{equation}
where $\Omega\subset\mathbb{R}^n$ is a bounded domain, $T>0$ and the operator $\mathcal{L}$ is defined by
$$
\mathcal{L}u(x,t)=\text{P.V.}\int_{\mathbb{R}^n}|u(x,t)-u(y,t)|^{p-2}(u(x,t)-u(y,t))K(x,y,t)\,dy,
$$
where P.V. stands for the principal value. 
We assume that the kernel $K$ is a measurable symmetric kernel with respect to $x$ and $y$ and satisfies
\begin{equation}\label{kernel}
\frac{\Lambda^{-1}}{|x-y|^{n+ps}}\leq K(x,y,t)\leq\frac{\Lambda}{|x-y|^{n+ps}},
\end{equation}
for almost every $x,y\in\R^n$ uniformly in $t\in(0,T)$ for some $\Lambda\geq 1$ and $s\in(0,1)$.  
If $K(x,y,t)=|x-y|^{-(n+ps)}$, then $\mathcal{L}$ reduces, up to a multiplicative constant, to the fractional $p$-Laplace operator $(-\Delta)_p^{s}$ and \eqref{maineqn} becomes
\begin{equation}\label{prot1}
\partial_t(|u|^{q-1}u)+(-\Delta)_p^{s}\,u=0.
\end{equation}
For $p=2$ and $q>0$, we have an equation of porous medium type. 

Our main result, Theorem \ref{thm1}, shows that a weak supersolution to \eqref{maineqn} is lower semicontinuous after a possible redefinition on a set of measure zero. 
In other words, a weak supersolution to \eqref{maineqn} has a lower semicontinuous representative. 
Furthermore, we prove that the lower semicontinuous representative can be recovered from the previous times, see Theorem \ref{thm2}.
This gives a pointwise interpretation for a weak supersolution at every point.
Our results applies to weak supersolutions and weak subsolutions which may change sign. 
In the nonlocal elliptic case, i.e.
\begin{equation}\label{ellipticcase}
\mathcal{L} u=0 \text{ in }\Omega,
\end{equation}
lower semicontinuity of weak supersolutions has been established by Korvenp\"{a}\"{a},  Kuusi and Palatucci \cite{Kuusilsc}.
As far as we are aware, our work contains the first results in the nonlocal parabolic case.
In the local case $s=1$, lower semicontinuity of weak supersolutions to nonlinear parabolic problems has been studied 
by Kinnunen and Lindqvist \cite{KinLin2, KLin}, Kuusi \cite{Kuusilsclocal}, Liao \cite{Liao} and Ziemer \cite{Ziemer}.

We apply a general result of Liao in \cite[Theorem 2.1]{Liao}, which asserts that every function that satisfies a measure theoretic property in Definition \ref{propertyP} has a lower semicontinuous representative.
In \cite{Liao} this result was applied for a large class of elliptic and parabolic partial differential equations in the local case $s=1$, but we show that this technique
can also be applied  in the nonlocal case $s\in(0,1)$. 
Our argument is based on a new energy estimate, Lemma \ref{Subenergyestimate}, along with De Giorgi type lemmas, see Lemma \ref{Lemma1} and Lemma \ref{Lemma2}.
In particular, the approach is independent of Harnack type estimates.
Lower semicontinuity results for weak supersolutions are needed, for example, in extending Perron's method in Korvenp\"{a}\"{a},  Kuusi and Palatucci \cite{Kuusilsc} to the parabolic nonlocal case,
but this will be discussed elsewhere.

%{\color{blue}A simple example
%\[
%u(x,t)=
%\begin{cases}
%1,&\quad t\ge0,\\
%0,&\quad t<0,
%\end{cases}
%\]
%shows that weak supersolutions are not necessarily continuous.
%As an immediate corollary to our main result we obtain that a weak solution to \eqref{maineqn} has a continuous representative, but our result does not give any estimate for the modulus of continuity.}
As far as we know, local H\"older continuity and Harnack's inequality for weak solutions to \eqref{maineqn} are open questions.
These questions have been studied extensively in the elliptic nonlocal case. 
For $p=2$, Kassmann in \cite{KassmanHarnack} proved a Harnack inequality for equations of the type \eqref{ellipticcase} with lower order perturbations.
This was extended by Di~Castro, Kuusi and  Palatucci \cite{Kuusiharnack} to the case $1<p<\infty$. 
See also Brasco,  Lindgren and  Schikorra \cite{BLS} and Di~Castro, Kuusi and Palatucci \cite{Kuusilocal}
for H\"older continuity results of \eqref{ellipticcase}. 

When $q=1$, the doubly nonlinear equation in \eqref{prot1} reduces to the fractional parabolic $p$-Laplace equation
\begin{equation}\label{fplap}
\partial_t u+(-\Delta)_p^{s}u=0.
\end{equation}
For such an equation, Str\"omqvist in \cite{Martinlocal} established local boundedness estimate for subsolutions and 
Brasco, Lindgren and Str{\"o}mqvist \cite{Martincont} 
proved H\"older continuity of weak solutions to \eqref{fplap} for $p\geq 2$. 
In the case when $p=2$, i.e. for equations modeled on $\partial_t u+(-\Delta)^s u=0$, 
Harnack's inequality has been established by Felsinger and Kassmann in  \cite{Kassweakharnack}, see also Str\"{o}mqvist \cite{Martinharnack}.  
Finally we mention that qualitative properties for fractional porous medium equations of the type,
\begin{equation*}\label{fracporous}
\partial_t u+(-\Delta)^s(u^m)=0,\quad m>0,
\end{equation*}
has been studied by Bonforte, Sire and V\'{a}zquez \cite{Sire}, Bonforte and V\'azquez \cite{BonVaz}, de~Pablo, Quir\'{o}s,  Rodr\'{\i}guez and V\'{a}zquez \cite{Vazquezpor1}.
%We note that the equation in \eqref{maineqn} is homogeneous for $q=p-1$ and nonhomogeneous when $q\neq p-1$.  
When $q=p-1$, equation \eqref{prot1} becomes the doubly nonlinear equation
\begin{equation*}\label{dblyfplap}
\partial_t(|u|^{p-2}u)+(-\Delta)_p^{s}u=0.
\end{equation*}
This equation is homogeneous in the sense that the class of solutions is closed under multiplication by constants.
For such an equation, local boundedness estimate  for positive subsolutions and  a reverse H\"older inequality for positive supersolutions has been proven 
by Banerjee, Garain and Kinnunen \cite{BGK}. 

This article is organized as follows. In Section 2, we introduce  basic notation, gather  some preliminary results  that are relevant to our work  and then state the main results. 
In Section 3, we prove  an energy estimate  which is  needed to derive the measure theoretic De Giorgi type lemmas in Section 4  using which we prove our main results. 

\section{Preliminaries and main results}
We begin with some known results for fractional Sobolev spaces, see \cite{Hitchhiker'sguide} for more details.
\begin{Definition}
Let $1<p<\infty$ and $0<s<1$ and assume that $\Omega\subset\mathbb{R}^n$ is an open and connected subset of $\mathbb R^n$. 
The fractional Sobolev space $W^{s,p}(\Omega)$ is defined by
$$
W^{s,p}(\Omega)=\Big\{u\in L^p(\Omega):\frac{|u(x)-u(y)|}{|x-y|^{\frac{n}{p}+s}}\in L^p(\Omega\times \Omega)\Big\}
$$
and it is endowed with the norm
$$
\|u\|_{W^{s,p}(\Omega)}=\Big(\int_{\Omega}|u(x)|^p\,dx+\int_{\Omega}\int_{\Omega}\frac{|u(x)-u(y)|^p}{|x-y|^{n+ps}}\,dx\,dy\Big)^\frac{1}{p}.
$$
{The fractional Sobolev space with zero boundary values is defined by}
$$
W_{0}^{s,p}(\Omega)={\big\{u\in W^{s,p}(\mathbb{R}^n):u=0\text{ in }\mathbb{R}^n\setminus\Omega\big\}}.
$$
\end{Definition}

Both $W^{s,p}(\Omega)$ and $W_{0}^{s,p}(\Omega)$ are reflexive Banach spaces, see \cite{Hitchhiker'sguide}. 
For any $v,k\in\mathbb{R}$, we denote the positive and negative parts of $(v-k)$ by 
\[
(v-k)_{+}=\max\{(v-k),0\}
\quad\text{and}\quad 
(v-k)_{-}=\max\{-(v-k),0\}.
\]
We note that  $(v-k)_{-}=(k-v)_{+}$. For any $a,b\in\mathbb{R}$, we have $|a_{+}-b_{+}|\leq|a-b|$ which implies $u_{+}\in W^{s,p}(\Omega)$ when $u\in W^{s,p}(\Omega)$. 
Analogously, we have $u_{-}\in W^{s,p}(\Omega)$.

The parabolic Sobolev space $L^p(0,T;W^{s,p}(\Omega))$, $T>0$,
consists of measurable functions $u$ on $\Omega\times(0,T)$ such that 
$$
||u||_{L^p(0,T;W^{s,p}(\Omega))}=\Big(\int_{0}^{T} ||u(\cdot,t)||^p_{W^{s,p}(\Omega)}\,dt\Big)^\frac{1}{p}<\infty.
$$
The local space $L^p_{\loc}(0,T;W^{s,p}_{\loc}(\Omega))$ is defined by requiring the conditions above for every $\Omega' \times [t_1,t_2]\Subset \Omega \times (0,T)$. 

Next we state a  Sobolev embedding theorem, see \cite{Hitchhiker'sguide}.
%We write by $C$ to denote a positive constant which may vary from line to line or even in the same line depending on the situation. 

\begin{Theorem}\label{ellipticembedding1}
Let $1<p<\infty$ and $0<s<1$ with $ps<n$ and $\kappa^{*}=\frac{n}{n-ps}$. 
There exists a constant $C=C(n,p,s)$ such that 
$$
\|u\|^p_{L^{\kappa^{*} p}(\mathbb{R}^n)}
\leq C\int_{\mathbb{R}^n}\int_{\mathbb{R}^n}\frac{|u(x)-u(y)|^p}{|x-y|^{n+ps}}\,dx\,dy
$$
for every $u\in W^{s,p}(\mathbb{R}^n)$.
If $\Omega$ is  a bounded  $W^{s,p}$-extension domain, there exists a constant $C=C(n, s, p, \Omega)$ such that, for any $\kappa \in [1, \kappa^*]$, we have
$$
\|u\|_{L^{\kappa p}(\Omega)}\leq C||u||_{W^{s,p}(\Omega)}
$$
for every $u\in W^{s,p}(\Omega)$.\\
If $ps=n$, then the above inequalities hold for any $\kappa\in[1,\infty)$. For $ps>n$, the second inequality holds for any $\kappa\in[1,\infty]$.
\end{Theorem}

Let $B_r(x_0)=\{x\in\mathbb R^n:|x-x_0|<r\}$ denote the ball in $\mathbb{R}^n$ of radius $r>0$ and center at $x_0\in\mathbb{R}^n$. 
The barred integral sign denotes the corresponding integral average.
As an application of Theorem \ref{ellipticembedding1} we obtain the following Sobolev type inequality, see \cite[Lemma 2.1]{Martinlocal} for a proof.

\begin{Lemma}\label{ellipticembedding2}
{Let $1<p<\infty$ and $0<s<1$.}
Assume that $u\in W^{s,p}(B_r)$, where $B_r=B_r(x_0)$. Let 
$\kappa^{*}=\frac{n}{n-ps}$, if $ps<n$.
There exists a constant $C=C(n,p,s)$ such that for every $\kappa\in[1,\kappa ^{*}]$, we have
\begin{equation}\label{Sobolevineq}
\Big(\fint_{B_r}|u(x)|^{\kappa p}\,dx\Big)^\frac{1}{\kappa}
\leq Cr^{ps-n}\int_{B_r}\int_{B_r}\frac{|u(x)-u(y)|^p}{|x-y|^{n+ps}}\,dx\,dy+C\fint_{B_r}|u(x)|^p\,dx.
\end{equation}
If $ps\geq n$, let $N>0$ be such that $n\leq ps<N$. Then the inequality \eqref{Sobolevineq}  holds for every $\kappa\in[1,\kappa^*]$ where $\kappa^*=\frac{N}{N-ps}$.
\end{Lemma}

The next inequality is a straightforward consequence of H\"older's inequality.

\begin{Lemma}\label{parabolic embedding}
Let $1<p<\infty$ and $0<s<1$. Assume that $u\in L^p(t_1,t_2;W^{s,p}(B_r))\cap L^{\infty}(t_1,t_2;L^m(B_r))$ where $m\geq 1$. 
Let $\kappa^*=\frac{n}{n-ps}$, if $ps<n$. Then for $l=p(1+\frac{ms}{n})$, we have 
\begin{equation}\label{Holderapp}
\int_{t_1}^{t_2}\int_{B_r}|u|^l\,dx dt\leq
\int_{t_1}^{t_2}\Bigl(\int_{B_r}|u|^{p\kappa^{*}}\,dx\Bigr)^\frac{1}{\kappa^{*}}\,dt
\Bigl(\esssup_{t_1<t<t_2}\int_{B_r}|u|^m\,dx\Bigr)^\frac{ps}{n}.
\end{equation}
If $ps\geq n$, let $N>0$ be such that $n\leq ps<N$. Then the inequality \eqref{Holderapp} holds with $n$, $l$ and $\kappa^*$ replaced by $N$, $l=p(1+\frac{ms}{N})$ and  $\kappa^*=\frac{N}{N-ps}$ respectively.
\end{Lemma}

We fix notation that will be used throughout the rest of the paper. 
%We refer to the set of numbers $\{n,p,q,s,\Lambda\}$ as the data.
Let $\Omega'\subset\Omega$ and $0\le t_1<t_2\le T$. 
We define the parabolic boundary of a space-time cylinder $\Omega'_{t_1,t_2}:=\Omega'\times(t_1,t_2)$ by 
$$
\partial_p \Omega'_{t_1,t_2}:=(\overline{\Omega'}\times\{t_1\})\cup (\partial \Omega'\times[t_1,t_2]).
$$

Let $r,\theta>0$ and $(x_0,t_0)\in\mathbb{R}^{n}\times\mathbb{R}$. Then we denote the backward and forward in time cylinders by
\begin{equation}\label{cylinder}
\begin{split}
\mathcal{Q}^{-}(x_0,t_0;r,\theta):=B_r(x_{0})\times(t_0-\theta r^{sp},t_0],\\
\mathcal{Q}^{+}(x_0,t_0;r,\theta):=B_r(x_0)\times[t_0,t_0+\theta r^{sp}),
\end{split}
\end{equation}
respectively. 
The centered cylinder is denoted by 
\begin{equation}\label{cyl}
\mathcal{Q}_r(x_0, t_0;\theta):= \mathcal{Q}^{-}(x_0,t_0;r,\theta) \cup \mathcal{Q}^{+}(x_0,t_0;r,\theta).\end{equation}
 When $(x_0,t_0)=(0,0)$, we simply denote $\mathcal{Q}^{-}(x_0,t_0;r,\theta)$ and $\mathcal{Q}^{+}(x_0,t_0;r,\theta)$ by   $\mathcal{Q}^{-}_r(\theta)$ and $\mathcal{Q}^{+}_r(\theta)$, respectively. 
 Also when $\theta =1$, we denote $\mathcal{Q}_r(x_0,t_0;\theta)$ by $\mathcal{Q}_r(x_0,t_0)$. 
 
For short, we denote 
\[
\mathcal{A}(u(x,y,t)):=|u(x,t)-u(y,t)|^{p-2}(u(x,t)-u(y,t))
\quad\text{and}\quad
d\mu:=K(x,y,t)\,dx\,dy.
\]

Next we define the notion of weak solution of the doubly nonlinear parabolic nonlocal equation in \eqref{maineqn}.
 
\begin{Definition}\label{subsupsolution}
A function $u\in L^\infty(0,T;L^{\infty}(\mathbb{R}^n))$ is a weak supersolution (subsolution) to \eqref{maineqn}
 if $u\in L^{q+1}_{\loc}(0,T;L^{q+1}_{\loc}(\Omega))\cap L^p_{\loc}(0,T;W_{\loc}^{s,p}(\Omega))$ and for every $\Omega' \times [t_1,t_2]\Subset \Omega \times (0,T)$, and  nonnegative test function $\phi\in W^{1,q+1}_{\loc}(0,T;L^{q+1}(\Omega'))\cap L^p_{\loc}(0,T;W_{0}^{s,p}(\Omega'))$, we have
\begin{equation}\label{weaksubsupsoln}
\begin{gathered}
\int_{t_1}^{t_2}\int_{\mathbb{R}^n}\int_{\mathbb{R}^n}\mathcal{A}(u(x,y,t)){(\phi(x,t)-\phi(y,t))}\,d\mu\,dt-\int_{t_1}^{t_2}\int_{\Omega'}|u(x,t)|^{q-1}u(x,t)\partial_t\phi(x,t)\,dx\,dt\\+
\int_{\Omega'}|u(x,t_2)|^{q-1}u(x,t_2)\phi(x,t_2)\,dx
-\int_{\Omega'}|u(x,t_1)|^{q-1}u(x,t_1)\phi(x,t_1)\,dx
\geq 0\,(\leq0).
\end{gathered}
\end{equation}
We say that $u$ is a weak solution if the integral in \eqref{weaksubsupsoln} is zero for every test function without a sign restriction.
\end{Definition}

\begin{Remark}\label{Lebpt}
The boundary terms in \eqref{weaksubsupsoln} are understood in the sense that
$$
\int_{\Omega'}|u(x,t_1)|^{q-1}u(x,t_1)\phi(x,t_1)\,dx=\lim_{h\to 0}\fint_{t_1}^{t_1+h}\int_{\Omega'}|u(x,t)|^{q-1}u(x,t) \phi(x,t)\,dx\,dt,
$$
and
$$
\int_{\Omega'}|u(x,t_2)|^{q-1}u(x,t_2)\phi(x,t_2)\,dx=\lim_{h\to 0}\fint_{t_2-h}^{t_2}\int_{\Omega'}|u(x,t)|^{q-1}u(x,t)\phi(x,t)\,dx\,dt.
$$
\end{Remark}

\begin{Remark}\label{Solsubsup}
From Definition \ref{subsupsolution}, it follows that $u$ is a weak supersolution of \eqref{maineqn} if and only if $-u$ is a weak subsolution of \eqref{maineqn}.
Moroever, $u$ is a weak solution if it is both a weak supersolution and a weak subsolution.
\end{Remark}

\begin{Remark}\label{NeedofLinfinity}
The assumption $u\in L^\infty(0,T;L^\infty(\mathbb{R}^n))$ ensures that  the first term in the left hand side of \eqref{weaksubsupsoln} is finite. \end{Remark}

\begin{Remark}\label{Molifier}
In the next section, we prove energy estimates for a weak supersolution $u$, where we use test functions depending on $u$ itself.
The time derivative $u_t$ can be justified by using a mollification
\begin{equation}\label{Molifierdef}
f_h(x,t)=\frac{1}{h}\int_{0}^{t}e^{\frac{s-t}{h}}f(x,s)\,ds.
\end{equation}
 in time (see \cite{BGK, Verenacontinuity, KLin}).
\end{Remark}

%\subsection*{Lower and upper semicontinuity}

%Let $D$ be an open subset of $\mathbb{R}^{n+1}$.
Let $u$ be a measurable function which is locally essentially bounded below in $\Omega\times(0,T)$. 
We define the  lower seimicontinuous regularization $u_*$ of $u$ as 
$$
u_*(x,t):=\essliminf_{(y,\hat{t})\to (x,t)}\,u(y,\hat{t})=\lim_{r\to 0}\essinf_{\mathcal{Q}_r(x,t;\theta)}\,u
$$
for every $(x,t)\in \Omega\times(0,T)$.
Analogously, for a locally essentially bounded above measurable function $u$ in $\Omega\times(0,T)$, we define an upper semicontinuous regularization $u^*$  of $u$ by
$$
u^*(x,t):=\esslimsup_{(y,\hat{t})\to (x,t)}\,u(y,\hat{t})=\lim_{r\to 0}\esssup_{\mathcal{Q}_r(x,t;\theta)}\,u
$$
for every $(x,t)\in\Omega\times(0,T)$.
It is easy to see that $u_*$ is lower semicontinuous and  $u^*$ is upper semicontinuous in $\Omega\times(0,T)$.

The following measure theoretic property from \cite{Liao} will be useful for us.

\begin{Definition}\label{propertyP}
Let $u$ be a measurable function which is locally bounded from below in $\Omega\times(0,T)$ and let 
\[
\mu_{-}\leq\essinf_{\mathcal{Q}_r(x_0,t_0;\theta)}\,u.
\]
Moreover, let $a, c \in (0,1)$ and $M>0$. 
We say that $u$ satisfies the property $(\mathcal D)$ if there exists a constant $\tau\in(0,1)$, 
which depends only on $a,M,\mu_{-}$ and other data, but independent of $r$, such that
$$
|\{u\leq\mu_{-}+M\}\cap\mathcal{Q}_r(x_0,t_0;\theta)|\leq\tau|\mathcal{Q}_r(x_0,t_0;\theta)|
$$
implies 
$$
u\geq\mu_{-}+aM\text{ a.e. in }\mathcal{Q}_{cr}(x_0,t_0;\theta).
$$
\end{Definition}

Let $u\in L^1_{\loc}(\Omega\times(0,T))$ and define
$$
\mathcal{F}:=\Big\{(x,t)\in\Omega\times(0,T):|u(x,t)|<\infty,\,\lim_{r\to 0}\fint_{\mathcal{Q}_r(x,t;\theta)}|u(x,t)-u(y,\hat{t})|\,dy\,d\hat{t}=0\Big\}.
$$
From the Lebesgue differentiation theorem we have $|\mathcal{F}|=|\Omega\times(0,T)|$.  
A result of Liao in \cite[Theorem 2.1]{Liao} shows that any such function with the property $(\mathcal D)$ has a lower semicontinuous representative. 

\begin{Theorem}\label{lscthm}
Let $u$ be a measurable function in $\Omega\times(0,T)$ which is locally essentially bounded below in $\Omega\times(0,T)$ and satisfies the property $(\mathcal D)$. 
Then $u(x,t)=u_*(x,t)$ for every $x\in \mathcal{F}$. 
In particular, $u_*$ is a lower semicontinuous representative of $u$ in $\Omega\times(0,T)$.
\end{Theorem}

%\subsection*{Statement of the main results.}
Next we state our main results. 
Our first result is a lower semicontinuity result for weak supersolutions. 
This follows from Lemma \ref{Lemma1} below, which asserts that weak supersolutions satisfy the property $(\mathcal D)$, and Theorem \ref{lscthm}.

\begin{Theorem}\label{thm1}
Let $1<p<\infty,\,q>0$ and $u$ be a weak supersolution of \eqref{maineqn} in $\Omega\times(0,T)$. Then $u_*(x,t)=u(x,t)$ at every Lebesgue point $(x,t)\in\Omega\times(0,T)$. 
In particular, $u_*$ is a lower semicontinuous representative of $u$ in $\Omega\times(0,T)$.
\end{Theorem}

As an immediate corollary we have the corresponding result for weak subsolutions.

\begin{Corollary}\label{thm3}
Let $1<p<\infty,\,q>0$ and $u$ be a weak subsolution of \eqref{maineqn} in $\Omega\times(0,T)$. 
Then $u^*(x,t)=u(x,t)$ at every Lebesgue point $(x,t)\in\Omega\times(0,T)$. In particular, $u^*$ is an upper semicontinuous representative of $u$ in $\Omega\times(0,T)$.
\end{Corollary}

The second result asserts that the lower semicontinuous representative, given by Theorem \ref{thm1}, is determined by previous times.
The proof is based on Lemma \ref{Lemma2} below and the proof of \cite[Theorem 3.1]{Liao}.

\begin{Theorem}\label{thm2}
Let $1<p<\infty,\,q>0$ and $u$ be a weak supersolution of \eqref{maineqn} in $\Omega\times(0,T)$ and $u_*$ is the lower semicontinuous representative of $u$ given by Theorem \ref{thm1}. 
Then there exists $\theta>0$ such that for every $(x,t)\in\Omega\times(0,T)$, we have
$$
u_*(x,t)=\lim_{r\to 0}\essinf_{\mathcal{Q}_r'(x,t,\theta)}\,u,
$$
where $\mathcal{Q}_r'(x,t,\theta)=B_r(x)\times(t-2\theta r^{ps},t-\theta r^{ps})$.
As a consequence,
$$
u_*(x,t)={\essliminf_{(y,\hat{t})\to(x,t),\,\hat{t}<t}}\,u(y,\hat{t})
$$
at every point $(x,t)\in\Omega\times(0,T)$.
\end{Theorem}

\section{The energy estimate}

In order to be able to consider sign changing functions, for $a\in\mathbb{R}$ and $\beta>0$, we define
\begin{equation}\label{power}
a^{\beta}=
\begin{cases}
|a|^{\beta-1}a, & a\neq 0,\\
0, &a=0.
\end{cases}
\end{equation}
In addition, for $a,k\in\mathbb{R}$ and $q>0$, we define a nonnegative auxiliary function $\xi$ by
\begin{equation}\label{xi}
\xi((a-k)_{-})= q\int_{a}^{k}|\eta|^{q-1}(\eta-k)_{-}\,d\eta.
\end{equation}
For more applications of this kind of function in the doubly nonlinear context, we refer to \cite{Verenacontinuity, GVespri, Kuusiholder, Liao}. 

The following elementary inequality will be useful for us, see \cite[Lemma 3.1]{Kuusilocal}.

\begin{Lemma}\label{epsilon}
Let $p\geq 1$ and $\gamma\in(0,1]$. 
For every $a,b\in\mathbb{R}^n$, we have 
$$
|a|^p\leq |b|^p+C(p)\gamma|b|^p+\big(1+C(p)\gamma\big)\gamma^{1-p}|a-b|^p,
$$
where $C(p)=(p-1)\Gamma(\max\{1,p-2\})$ and $\Gamma$ denotes the gamma function.
\end{Lemma}

To prove Theorem \ref{thm1} and Theorem \ref{thm2}, we apply  the following Caccioppoli type estimate for weak supersolutions.

\begin{Lemma}\label{Subenergyestimate}
Let $1<p<\infty,\,q,\,l,>0,\,k\in\mathbb{R}$ and $x_0\in\mathbb{R}^n$ be such that $B_r\times(t_0-l,t_0):=B_r(x_0)\times(t_0-l,t_0)\Subset\Omega\times(0,T)$. Assume that $u$ is a weak supersolution of \eqref{maineqn} in $\Omega\times(0,T)$.
Then there exists a constant {$C=C(p,\Lambda)$} such that 
\begin{align*}
&\int_{t_0-l}^{t_0}\int_{B_r}\int_{B_r}|(u-k)_{-}(x,t)\psi(x,t)-(u-k)_{-}(y,t)\psi(y,t)|^p \,d\mu\,dt\\
&\qquad+\esssup_{t_0-l<t<t_0}\int_{B_r}\psi(x,t)^p\xi((u-k)_{-}(x,t))\,dx\\&\leq C\Bigg(\int_{t_0-l}^{t_0}\int_{B_r}\int_{B_r}{\max\{(u-k)_{-}(x,t),(u-k)_{-}(y,t)\}^p|\psi(x,t)-\psi(y,t)|^p}\,d\mu\,dt\\
&\qquad+\esssup_{(x,t)\in\mathrm{supp}\,\psi,\,t_0-l<t<t_0}\int_{{\mathbb{R}^n\setminus B_r}}{\frac{(u-k)_{-}(y,t)^{p-1}}{|x-y|^{n+ps}}}\,dy
\int_{t_0-l}^{t_0}\int_{B_r}(u-k)_{-}(x,t)\psi(x,t)^p\,dx\,dt\\
&\qquad+\int_{t_0-l}^{t_0}\int_{B_r}\xi((u-k)_{-}(x,t))\partial_t\psi(x,t)^p\,dx\,dt+\int_{B_r\times\{t_0-l\}}\psi(x,t)^{p}\xi((u-k)_{-}(x,t))\,dx\Bigg),
\end{align*}
for every nonnegative, piecewise smooth cutoff function $\psi$ vanishing on $\partial B_r\times(t_0-l,t_0)$.
Here $\xi$ is the auxiliary function defined in \eqref{xi}.
\end{Lemma}

\begin{proof}
 We denote by $t_1=t_0-l$, $t_2=t_0$ and $w(x,t)=(u-k)_{-}(x,t)=(k-u)_{+}(x,t)$. Following \cite{Verenacontinuity}, for fixed $t_1<l_1<l_2<t_2$ and $\epsilon>0$ small enough, we define a
 Lipschitz continuous cutoff function $\zeta_{\epsilon}:[t_1,t_2]\to[0,1]$ by
\[
  \zeta_{\epsilon}(t) :=
  \begin{cases}
    0, &  t_1\leq t\leq l_1-\epsilon,\\
    1+\frac{t-l_1}{\epsilon}, & l_1-\epsilon<t\leq l_1, \\
    1, &  l_1<t\leq l_2, \\
    1-\frac{t-l_{2}}{\epsilon}, & l_2<t\leq l_2+\epsilon, \\
    0, &  l_2 +\epsilon<t\leq t_2.
  \end{cases}
\]

Choose 
\[
\phi(x,t)=w(x,t)\psi(x,t)^p\zeta_{\epsilon}(t)
\] 
as a test function in \eqref{weaksubsupsoln}. 
Keeping in mind \eqref{power}, we denote 
\[
v_{h}^{q}(x,t)=(u^{q})_h(x,t)
\quad\text{and}\quad
\mathcal{V}(u(x,y,t))=\mathcal{A}(u(x,y,t))K(x,y,t),
\]
where $(\cdot)_h$ is the mollification as defined in \eqref{Molifierdef}.
Following \cite{BGK, Verenacontinuity,KLin}, we conclude that
\begin{equation}\label{mollified}
\lim_{\epsilon\to 0}\lim_{h\to 0}(I_{h,\epsilon}+J_{h,\epsilon})\geq 0,
\end{equation}
where
$$
I_{h,\epsilon}=\int_{t_1}^{t_2}\int_{B_r}\partial_t{v_{h}^{q}}(x,t)\phi(x,t)\,dx\,dt
=\int_{t_1}^{t_2}\int_{B_r}\psi(x,t)^p\zeta_{\epsilon}(t)\partial_t{v_{h}^{q}}(x,t)w(x,t)\,dx\, dt,
$$
and
\begin{align*}
J_{h,\epsilon}&=\int_{t_1}^{t_2}\int_{\mathbb{R}^n}\int_{\mathbb{R}^n}\big(\mathcal{V}(u(x,y,t))\big)_{h}(\phi(x,t)-\phi(y,t))\,dx\,dy\,dt\\
&=\int_{t_1}^{t_2}\int_{\mathbb{R}^n}\int_{\mathbb{R}^n}\big(\mathcal{V}(u(x,y,t))\big)_{h}\big(w(x,t)\psi(x,t)^p-w(y,t)\psi(y,t)^p\big)\zeta_{\epsilon}(t)\,dx\,dy\,dt.
\end{align*}
\textbf{Estimate of $I_{h,\epsilon}$:} Proceeding similarly as in the proof of \cite[Proposition 3.1]{Verenacontinuity}, we may rewrite
\begin{equation}\label{Iheps}
\begin{split}
I_{h,\epsilon}=\int_{t_1}^{t_2}\int_{B_r}\psi(x,t)^p\zeta_{\epsilon}(t)\partial_t{v_{h}^{q}}(x,t)
\big(w(x,t)-(k-v_h)_{+}(x,t)\big)\,dx\,dt\\+\int_{t_1}^{t_2}\int_{B_r}\psi(x,t)^p\zeta_{\epsilon}(t)\partial_t{v_{h}^{q}}(x,t)(k-v_h)_{+}(x,t)\,dx\,dt.
\end{split}
\end{equation}
By the properties of mollifiers, we have
$$
\partial_t v_h^{q}=\frac{u^{q}-v_h^{q}}{h}.
$$
Then by using  the fact that the mapping $x\to (k-x^\frac{1}{q})_{+}$ is monotone decreasing and also by noting that $w= (k-u)_{+}$, 
we conclude that the first integral in \eqref{Iheps} is nonpositive and thus
\begin{equation}\label{Iheps1}
\begin{split}
I_{h,\epsilon}&\leq \int_{t_1}^{t_2}\int_{B_r}\psi(x,t)^p\zeta_{\epsilon}(t)\partial_t{v_{h}^{q}}(x,t)(k-v_h)_{+}(x,t)\,dx\,dt\\
&=-\int_{t_1}^{t_2}\int_{B_r}\psi(x,t)^p\zeta_{\epsilon}(t)\partial_t\xi\bigl((k-v_h)_+\bigr)\,dx\,dt\\
&=\int_{t_1}^{t_2}\int_{B_r}\bigl(\psi(x,t)^p\zeta_{\epsilon}'(t)+\partial_t\psi(x,t)^p\zeta_{\epsilon}(t)\bigr)\xi\bigl((k-v_h)_+(x,t)\bigr)\,dx\,dt,
\end{split}
\end{equation}
where we have used the fact that
$$
\partial_t \xi((k-v_h)_+)=-\partial_t v_h^{q}(k-v_h)_{+}.
$$
By passing to the limit as $h\to 0$ in \eqref{Iheps1}, we obtain
\begin{equation}\label{Iheps2}
\begin{split}
\lim_{h\to 0}I_{h,\epsilon}&\leq \int_{t_1}^{t_2}\int_{B_r}\bigl(\psi(x,t)^p\zeta_{\epsilon}'(t)+\partial_t\psi(x,t)^p\zeta_{\epsilon}(t)\bigr)\xi(w)\,dx\,dt.
\end{split}
\end{equation}
Then by letting $\epsilon\to 0$ in \eqref{Iheps2}, we have
\begin{equation}\label{Iheps3}
\begin{split}
\lim_{\epsilon\to 0}\lim_{h\to 0} I_{h,\epsilon}\leq \int_{B_r}\psi(x,l_1)^{p}\xi(w(x,l_1))\,dx-\int_{B_r}\psi(x,l_2)^{p}\xi(w(x,l_2))\,dx\\+\int_{l_1}^{l_2}\int_{B_r}\partial_t\psi(x,t)^p\xi(w(x,t))\,dx\,dt.
\end{split}
\end{equation}

\textbf{Estimate of $J_{h,\epsilon}$:} Following the proof of \cite[Lemma 3.1]{BGK}, we obtain
\begin{equation}\label{Jheps}
\begin{split}
\lim_{h\to 0}J_{h,\epsilon}
&=J_{\epsilon}\\
&=\int_{t_1}^{t_2}\int_{\mathbb{R}^n}\int_{\mathbb{R}^n}\mathcal{V}(u(x,y,t))\big(w(x,t)\psi(x,t)^p-w(y,t)\psi(y,t)^p\big)\zeta_{\epsilon}(t)\,dx\,dy\,dt\\
&=J^1_{\epsilon}+J^2_{\epsilon},
\end{split}
\end{equation}
where
\[
J^{1}_{\epsilon}=\int_{t_1}^{t_2}\int_{B_r}\int_{B_r}\mathcal{A}(u(x,y,t))\bigl(w(x,t)\psi(x,t)^p-w(y,t)\psi(y,t)^p\bigr)\zeta_{\epsilon}(t)\,d\mu\,dt,
\]
and
\[
J^{2}_{\epsilon}=2\int_{t_1}^{t_2}\int_{\mathbb{R}^n\setminus B_r}\int_{B_r}\mathcal{A}(u(x,y,t))w(x,t)\psi(x,t)^p\zeta_{\epsilon}(t)\,d\mu\,dt.
\]
\textbf{Estimate of $J^{1}_{\epsilon}$:} To estimate the integral $J^{1}_{\epsilon}$, we use some ideas  from the  proof of \cite[Theorem 1.4]{Kuusilocal}. By symmetry we may assume $u(x,t)<u(y,t)$. In this case, for every fixed $t$, we observe that
\begin{align*}
&|u(x,t)-u(y,t)|^{p-2}(u(x,t)-u(y,t))\big(w(x,t)\psi(x,t)^p-w(y,t)\psi(y,t)^p\big)\\
&=-(u(y,t)-u(x,t))^{p-1}\big(w(x,t)\psi(x,t)^p-w(y,t)\psi(y,t)^p\big)\\
&=\begin{cases}
-(w(x,t)-w(y,t))^{p-1}\big(w(x,t)\psi(x,t)^p-w(y,t)\psi(y,t)^p\big),\text{ if }u(x,t)<k,\,u(y,t)<k,\\
-(u(y,t)-u(x,t))^{p-1}w(x,t)\psi(x,t)^p,\text{ if }u(x,t)\leq k,\,u(y,t)>k,\\
0,\text{ otherwise,}
\end{cases}
\end{align*}
where we used the fact that $u(x,t)<u(y,t)$ implies $w(x,t)>w(y,t)$. Therefore in all cases we have
\begin{align*}
&|u(x,t)-u(y,t)|^{p-2}(u(x,t)-u(y,t))\big(w(x,t)\psi(x,t)^p-w(y,t)\psi(y,t)^p\big)\\
&\qquad\leq -(w(x,t)-w(y,t))^{p-1}\bigl(w(x,t)\psi(x,t)^p-w(y,t)\psi(y,t)^p\bigr).
\end{align*}
This implies
\[
J^{1}_{\epsilon}\leq -\int_{t_1}^{t_2}\int_{B_r}\int_{B_r}(w(x,t)-w(y,t))^{p-1}(w(x,t)\psi(x,t)^p-w(y,t)\psi(y,t)^p)\zeta_{\epsilon}(t)\,d\mu\, dt.
\]
Let us first consider the case when $\psi(x,t)\leq\psi(y,t)$. 
Choosing $a=\psi(y,t)$ and $b=\psi(x,t)$ in Lemma \ref{epsilon} we obtain
\begin{equation}\label{epsilonequationinequality}
\psi(x,t)^p\geq (1-C(p)\gamma)\psi(y,t)^p-(1+C(p)\gamma)\gamma^{1-p}|\psi(x,t)-\psi(y,t)|^p
\end{equation}
for any $\gamma\in(0,1]$, where $C(p)=(p-1)\Gamma(\max\{1,p-2\})$. 
By letting
$$
\gamma=\frac{1}{\max\{1,2C(p)\}}\frac{w(x,t)-w(y,t)}{w(x,t)}\in(0,1],
$$
we deduce from \eqref{epsilonequationinequality} that there exists a positive constant $C=C(p)$ such that
\begin{align*}
(w(x,t)&-w(y,t))^{p-1}w(x,t)\psi(x,t)^p\\
&\geq (w(x,t)-w(y,t))^{p-1}w(x,t)\max\{\psi(x,t),\psi(y,t)\}^p\\
&\qquad-\frac{1}{2}(w(x,t)-w(y,t))^p\max\{\psi(x,t),\psi(y,t)\}^p\\
&\qquad-C\max\{w(x,t),w(y,t)\}^p|\psi(x,t)-\psi(y,t)|^p.
\end{align*}
Here we used the fact that under the assumptions $\psi(x,t)\leq\psi(y,t)$ and $w(x,t)>w(y,t)$ we have $\max\{\psi(x,t),\psi(y,t)\}=\psi(y,t)$ and $\max\{w(x,t),w(y,t)\}=w(x,t)$. In the other cases $w(x,t)> w(y,t)$, $\psi(x,t)\geq\psi(y,t)$ or $w(x,t)=w(y,t)$, the above estimate is clear. Therefore, when $w(x,t)\geq w(y,t)$, we have
\begin{equation}\label{subballpartestimate1}
\begin{split}
(w(x,t)&-w(y,t))^{p-1}(w(x,t)\psi(x,t)^p-w(y,t)\psi(y,t)^p)\\
&\geq (w(x,t)-w(y,t))^{p-1}(w(x,t)\max\{\psi(x,t),\psi(y,t)\}^p-w(y,t)\psi(y,t)^p)\\
&\qquad-\frac{1}{2}(w(x,t)-w(y,t))^p\max\{\psi(x,t),\psi(y,t)\}^p\\
&\qquad-C\max\{w(x,t),w(y,t)\}^p|\psi(x,t)-\psi(y,t)|^p\\
&\geq \frac{1}{2}(w(x,t)-w(y,t))^p\max\{\psi(x,t),\psi(y,t)\}^p\\
&\qquad-C\max\{w(x,t),w(y,t)\}^p|\psi(x,t)-\psi(y,t)|^p,
\end{split}
\end{equation}
where $C=C(p)$.
If $w(x,t)<w(y,t)$, we may interchange the roles of $x$ and $y$ above in order to obtain \eqref{subballpartestimate1}. 
We observe that 
\begin{equation}\label{subballpartestimate2}
\begin{gathered}
|w(x,t)\psi(x,t)-w(y,t)\psi(y,t)|^p\leq 2^{p-1}|w(x,t)-w(y,t)|^{p}\max\{\psi(x,t),\psi(y,t)\}^p\\
+2^{p-1}\max\{w(x,t),w(y,t)\}^p|\psi(x,t)-\psi(y,t)|^p.
\end{gathered}
\end{equation}
Thus from \eqref{subballpartestimate1} and \eqref{subballpartestimate2}, we deduce that
\begin{equation}\label{subballpartestimate3}
\begin{split}
J_{\epsilon}^{1}&\leq -c\int_{t_1}^{t_2}\int_{B_r}\int_{B_r}|w(x,t)\psi(x,t)-w(y,t)\psi(y,t)|^p\zeta_{\epsilon}(t)\,d\mu\, dt\\
&\qquad+C\int_{t_1}^{t_2}\int_{B_r}\int_{B_r}\max\{w(x,t),w(y,t)\}^p|\psi(x,t)-\psi(y,t)|^p\zeta_{\epsilon}(t)\,d\mu\, dt,
\end{split}
\end{equation}
for some positive constants $c=c(p)$ and  $C=C(p)$.\\
\textbf{Estimate of $J_{\epsilon}^{2}$:} To estimate $J_{\epsilon}^{2}$, let $x\in B_r$ and $y\in \mathbb{R}^n\setminus B_r$. 
If $u(x,t)\leq u(y,t)$, then $J_{\epsilon}^2\leq 0$. 
It remains to consider the case when $u(x,t)>u(y,t)$. 
If $u(y,t)<u(x,t)<k$, then 
\begin{align*}
|u(x,t)-u(y,t)|^{p-2}(u(x,t)-u(y,t))w(x,t)
&=(u(x,t)-u(y,t))^{p-1}w(x,t)\\
&\leq w(y,t)^{p-1}w(x,t).
\end{align*}
This implies
\begin{equation}\label{subballpartestimate4}
\begin{split}
J_{\epsilon}^{2}&\leq \int_{t_1}^{t_2}\int_{\mathbb{R}^n\setminus B_r}\int_{B_r}K(x,y,t)w(y,t)^{p-1}w(x,t)\psi(x,t)^p\zeta_{\epsilon}(t)\,dx\,dy\,dt\\
&\leq\Lambda\esssup_{(x,t)\in\mathrm{supp}\,\psi,\,{t_1<t<t_2}}\int_{\mathbb{R}^n\setminus B_r}\frac{w(y,t)^{p-1}}{|x-y|^{n+ps}}\,dy
\int_{t_1}^{t_2}\int_{B_r}w(x,t)\psi(x,t)^p\zeta_{\epsilon}(t)\,dx\,dt.
\end{split}
\end{equation}
From \eqref{subballpartestimate3} and \eqref{subballpartestimate4}, we conclude that there exist positive constants $c=c(p)$ and $C=C(p)$, such that 
\begin{equation}\label{jlimit}
\begin{split}
&\lim_{\epsilon\to 0}\lim_{h\to 0}J_{h,\epsilon}
=\lim_{\epsilon\to 0}J_{\epsilon}
=\lim_{\epsilon\to 0}(J_{\epsilon}^{1}+J_{\epsilon}^{2})\\
&\leq -c\int_{l_1}^{l_2}\int_{B_r}\int_{B_r}|w(x,t)\psi(x,t)-w(y,t)\psi(y,t)|^p\,d\mu\, dt\\
&\qquad+C\int_{l_1}^{l_2}\int_{B_r}\int_{B_r}\max\{w(x,t),w(y,t)\}^p|\psi(x,t)-\psi(y,t)|^p\,d\mu\,dt\\
&\qquad+\Lambda\esssup_{(x,t)\in\mathrm{supp}\,\psi,\,{t_1<t<t_2}}\int_{\mathbb{R}^n\setminus B_r}\frac{w(y,t)^{p-1}}{|x-y|^{n+ps}}\,dy
\int_{l_1}^{l_2}\int_{B_r}w(x,t)\psi(x,t)^p\,dx\,dt.
\end{split}
\end{equation}
By inserting \eqref{Iheps3} and \eqref{jlimit} into \eqref{mollified} and first by letting $l_1\to t_1$ and then by letting $l_2\to t_2$, we get
\begin{equation}\label{seminorm1}
\begin{split}
&\int_{t_1}^{t_2}\int_{B_r}\int_{B_r}|w(x,t)\psi(x,t)-w(y,t)\psi(y,t)|^p\,d\mu\,dt\\
&\leq C(p,\Lambda)\Bigg(\int_{t_1}^{t_2}\int_{B_r}\int_{B_r}{\max\{w(x,t),w(y,t)\}^p|\psi(x,t)-\psi(y,t)|^p}\,d\mu\,dt\\
&\qquad+\esssup_{(x,t)\in\mathrm{supp}\,\psi,\,t_1<t<t_2}\int_{{\mathbb{R}^n\setminus B_r}}{\frac{w(y,t)^{p-1}}{|x-y|^{n+ps}}}\,dy
\int_{t_1}^{t_2}\int_{B_r}w(x,t)\psi(x,t)^p\,dx\,dt\\
&\qquad+\int_{t_1}^{t_2}\int_{B_r}\xi(w(x,t))\partial_t\psi(x,t)^p\,dx\,dt+\int_{B_r\times\{t_1\}}\psi(x,t)^p\xi(w(x,t))\,dx\Bigg).
\end{split}
\end{equation}
Again using \eqref{Iheps3} and \eqref{jlimit} and first  by letting $l_1\to t_1$ and then by choosing $l_2 \in (t_1, t_2)$ such that
\[
\int_{B_r}\xi(w(x,l_2))\psi(x,l_2)^p\,dx \geq \frac{1}{2} \esssup_{t_1 < t< t_2} \int_{B_r}\xi(w(x,t))\psi(x,t)^p\,dx,
\]
we obtain
\begin{equation}\label{parabolicnorm1}
\begin{split}
&\esssup_{t_1 < t< t_2} \int_{B_r}\xi(w(x,t))\psi(x,t)^p\,dx\\
&\leq C(p,\Lambda)\Bigg(\int_{t_1}^{t_2}\int_{B_r}\int_{B_r}{\max\{w(x,t),w(y,t)\}^p|\psi(x,t)-\psi(y,t)|^p}\,d\mu\,dt\\
&\qquad+\esssup_{(x,t)\in\mathrm{supp}\,\psi,\,t_1<t<t_2}\int_{{\mathbb{R}^n\setminus B_r}}{\frac{w(y,t)^{p-1}}{|x-y|^{n+ps}}}\,dy
\int_{t_1}^{t_2}\int_{B_r}w(x,t)\psi(x,t)^p\,dx\,dt\\
&\qquad+\int_{t_1}^{t_2}\int_{B_r}\xi(w(x,t))\partial_t\psi(x,t)^p\,dx\,dt+\int_{B_r\times\{t_1\}}\psi(x,t)^p\xi(w(x,t))\,dx\Bigg).
\end{split}
\end{equation}
The required estimate follows from \eqref{seminorm1} and \eqref{parabolicnorm1}.
\end{proof}

\section{Proof of the main results}

From \cite[Lemma 2.2]{Verenacontinuity}, we have the following bounds for the auxiliary function $\xi$ defined in \eqref{xi}.
\begin{Lemma}\label{Auxfnlemma}
Let $q>0$. Then there exists a constant $\lambda=\lambda(q)>0$ such that
$$
\frac{1}{\lambda}(|a|+|k|)^{q-1}(a-k)_{-}^2\leq \xi((a-k)_{-})\leq\lambda(|a|+|k|)^{q-1}(a-k)_{-}^2
$$
for every $a,k\in\mathbb{R}$.
\end{Lemma}

For the following real analysis lemma, see \cite[Lemma 4.1]{Dibe}.

\begin{Lemma}\label{Deg2}
Let $(Y_j)_{j=0}^{\infty}$ be a sequence of positive real numbers such that for some constants $c_0>0$, $b>1$ and $\beta>0$ we have
\[
Y_0\leq c_{0}^{-\frac{1}{\beta}}b^{-\frac{1}{\beta^2}}
\quad{and}\quad
Y_{j+1}\leq c_0 b^{j} Y_j^{1+\beta}
\quad\text{for every $j=0,1,2,\dots$}.
\] 
Then $Y_j\to 0$ as $j\to\infty$.
\end{Lemma}

We now prove a De Giorgi type lemma from which it follows that a weak supersolution satisfies the property $(\mathcal D)$ in Definition \ref{propertyP}.

\begin{Lemma}\label{Lemma1}
Let $1<p<\infty,\,q,\,L,\,\theta>0$, $a\in(0,1)$ and assume that $u$ is a weak supersolution of \eqref{maineqn} in $\Omega\times(0,T)$. 
Assume that 
\[
\mu_{-}\leq\essinf_{\mathbb{R}^n\times(0,T)}\,u
\] 
and let $\mathcal{Q}^{-}(x_0,t_0;r,\theta)\Subset\Omega\times(0,T)$. Then there exists a constant $\tau\in(0,1)$, depending only on $a,L,\mu_{-},\theta$ and the data, such that if 
$$
|\{u\leq\mu_{-}+L\}\cap \mathcal{Q}^{-}(x_0,t_0;r,\theta)|\leq\tau|\mathcal{Q}^{-}(x_0,t_0;r,\theta)|,
$$
then
$$
u\geq\mu_{-}+aL\text{ a.e. in } \mathcal{Q}^{-}\bigl(x_0,t_0;\frac{3r}{4},\theta\bigr).
$$
\end{Lemma}

\begin{proof} First we assume that $sp< n$. 
Without loss of generality, we may assume that $(x_0,t_0)=(0,0)$. 
For $j=0,1,2,\dots$, let
\begin{equation}\label{parameter1}
k_j=\mu_{-}+aL+\frac{(1-a)L}{2^j},
\quad
\bar{k}_j=\frac{k_j+k_{j+1}}{2},
\quad
r_j=\frac{3r}{4}+\frac{r}{2^{j+2}},
\quad
\bar{r}_j=\frac{r_j+r_{j+1}}{2},
\end{equation}
\begin{equation}\label{parameter2}
B_j=B_{r_j}(0),
\quad
\overline{B}_j=B_{\bar{r}_j}(0),
\quad
w_j=(k_j-u)_{+},
\quad
\bar{w}_j=(\bar{k}_j-u)_{+},
\end{equation}
and
\begin{equation}\label{parameter3}
\mathcal{Q}_j=B_j\times(-\theta r_j^{ps},0],
\quad
\overline{\mathcal{Q}}_j=\overline{B}_{j}\times(-\theta \bar{r}_{j}^{ps},0].
\end{equation}
We observe that $\mathcal{Q}_{j+1}\subset\overline{\mathcal{Q}}_j\subset \mathcal{Q}_j$, $\bar{k}_j<k_j$ and hence $\bar{w}_j\leq w_j$. 
Moreover, from Lemma \ref{Auxfnlemma}, we have
\begin{equation}\label{Auxfnlemma1}
\frac{1}{\lambda(q)}(|u|+|k_j|)^{q-1}w_j^{2}\leq\xi(w_j)\leq\lambda(q)(|u|+|k_j|)^{q-1}w_j^{2}.
\end{equation}
Let  $\psi_{1,j}\in C^{\infty}(B_j)$ and $\eta_{1,j}\in C^{\infty}((-\theta r_j^{ps},0])$ be  nonnegative  functions satisfying
%$0 \leq \psi_{1,j} \leq 1$, $\psi_{1,j}=1$ in $B_{j+1}$,
\[
0 \leq \psi_{1,j} \leq 1,\quad
\psi_{1,j}=1\text{ in $B_{j+1}$},\quad
\mathrm{dist}(\mathrm{supp}\,\psi_{1,j},\,\mathbb{R}^n\setminus B_j)\geq 2^{-j-1}r
\]
and $0 \leq \eta_{1,j} \leq1$, $\eta_{1,j} =1$ on $(-\theta \bar{r}_j^{ps}, 0]$.
Then  for $\psi_j=\psi_{1,j}\,\eta_{1,j}$ we have
$$
0\leq\psi_j\leq 1\text{ in }\mathcal{Q}_j,\quad
\psi_j=1\text{ in }\overline{\mathcal{Q}}_j,\quad
\psi_j=0\text{ on }\partial_p \mathcal{Q}_{j},
$$
and
$$
|\nabla \psi_j|\leq C\frac{2^j}{r},
\quad
|\partial_t \psi_j|\leq C\frac{2^{jps}}{\theta r^{ps}},
$$
for some constant $C=C(n,p,s)$.

By Lemma \ref{Subenergyestimate} applied to $w_j$, there exists a constant $C=C(p,\Lambda)$, such that
\begin{equation}\label{energyapp1}
\begin{split}
&\int_{-\theta{r_j}^{ps}}^{0}\int_{B_{j}}\int_{B_{j}}|w_j(x,t)\psi_{j}(x,t)-w_j(y,t)\psi_{j}(y,t)|^p \,d\mu\,dt\\
&\qquad+\esssup_{-\theta{r_j}^{ps}<t<0}\int_{B_{j}}\psi_{j}(x,t)^p\xi(w_j(x,t))\,dx\\&\leq C\Bigg(\int_{-\theta{r_j}^{ps}}^{0}\int_{B_{j}}\int_{B_{j}}{\max\{w_j(x,t),w_j(y,t)\}^p|\psi_{j}(x,t)-\psi_{j}(y,t)|^p}\,d\mu\,dt\\
&\qquad+\esssup_{(x,t)\in\mathrm{supp}\,\psi_j,\,-\theta r_j^{ps}<t<0}\int_{{\mathbb{R}^n\setminus B_{j}}}{\frac{w_j(y,t)^{p-1}}{|x-y|^{n+ps}}}\,dy
\int_{-\theta r_j^{ps}}^{0}\int_{B_{j}}w_j(x,t)\psi_j(x,t)^p\,dx\,dt\\
&\qquad+\int_{-\theta r_j^{ps}}^{0}\int_{B_{j}}\xi(w_j(x,t))\partial_t\psi_j(x,t)^p\,dx\,dt\Bigg).
\end{split}
\end{equation}
Now using \eqref{Auxfnlemma1}, the fact that $\bar{w}_j\leq w_j$ along with the properties of $\psi_j$ in \eqref{energyapp1}, we obtain
\begin{equation}\label{energyapp2}
\begin{split}
&\int_{-\theta{\bar{r}_j}^{ps}}^{0}\int_{\overline{B}_{j}}\int_{\overline{B}_{j}}|\bar{w}_j(x,t)-\bar{w}_j(y,t)|^p \,d\mu\,dt+\esssup_{-\theta{\bar{r}_j}^{ps}<t<0}\int_{\overline{B}_{_j}}(|u|+|k_j|)^{q-1}\bar{w}_j^2\,dx\\&\leq
C(n,p,q,s,\Lambda)(I_1+I_2+I_3),
\end{split}
\end{equation}
where
\begin{align*}
I_1&=\int_{-\theta{r_j}^{ps}}^{0}\int_{B_{j}}\int_{B_{j}}{\max\{w_j(x,t),w_j(y,t)\}^p|\psi_{j}(x,t)-\psi_{j}(y,t)|^p}\,d\mu\,dt,\\
I_2&=\esssup_{(x,t)\in\mathrm{supp}\,\psi_j,\,-\theta r_j^{ps}<t<0}\int_{{\mathbb{R}^n\setminus B_{j}}}{\frac{w_j(y,t)^{p-1}}{|x-y|^{n+ps}}}\,dy
\int_{-\theta r_j^{ps}}^{0}\int_{B_{j}}w_j(x,t)\psi_j(x,t)^p\,dx\,dt,\text{ and }\\
I_3&=\int_{-\theta r_j^{ps}}^{0}\int_{B_{j}}(|u|+|k_j|)^{q-1}{w}_j^2|\partial_t\psi_j(x,t)^p|\,dx\,dt.
\end{align*}
Let $A_j=\{u<k_j\}\cap \mathcal{Q}_j$. 
We estimate the terms $I_1$ , $I_2$ and $I_3$ separately.\\
\textbf{Estimate of $I_1$:} Using the facts that $\frac{r}{2}<r_j<r$, $w_j\leq L$, and the derivative bounds for $\psi_j$, we obtain
\begin{equation}\label{estI1}
\begin{split}
I_1&=\int_{-\theta{r_j}^{ps}}^{0}\int_{B_{j}}\int_{B_{j}}{\max\{w_j(x,t),w_j(y,t)\}^p|\psi_{j}(x,t)-\psi_{j}(y,t)|^p}\,d\mu\,dt\\
&\leq C\frac{2^{jp}}{r^{ps}}L^p|A_j|,
\end{split}
\end{equation}
where $C=C(n,p,s,\Lambda)$.\\
\textbf{Estimate of $I_2$:} For every $x\in\mathrm{supp}\,\psi_{1,j}$ and every $y\in\mathbb{R}^n\setminus B_j$, we have
\begin{equation}\label{I2}
\frac{1}{|x-y|}=\frac{1}{|y|}\frac{|x-(x-y)|}{|x-y|}\leq\frac{1}{|y|}(1+2^{j+3})\leq\frac{2^{j+4}}{|y|}.
\end{equation}
Using $r_j>\frac{r}{2}$, $0\leq\psi_j\leq 1$ along with the fact that 
$$
\mu_{-}\leq \essinf_{\mathbb{R}^n\times(0,T)}\,u,
$$
which gives $w_j\leq L$ in $\mathbb{R}^n\times(0,T)$, we obtain
\begin{equation}\label{estI2}
\begin{split}
I_2&\leq\esssup_{x\in\mathrm{supp}\,\psi_{1,j},\,-\theta r_j^{ps}<t<0}\int_{{\mathbb{R}^n\setminus B_{j}}}{\frac{w_j(y,t)^{p-1}}{|x-y|^{n+ps}}}\,dy
\int_{-\theta r_j^{ps}}^{0}\int_{B_{j}}w_j(x,t)\,dx\,dt\\
&\leq C\frac{2^{j(n+ps)}}{r^{ps}}L^p|A_j|,
\end{split}
\end{equation}
where $C=C(n,p,s)$.\\
\textbf{Estimate of $I_3$:}
When $0<q\leq 1$, using the inequality $|u|+|k_j|\geq w_j$ along with the fact that $w_j\leq L$, we get
\begin{equation}\label{estI3less}
\begin{split}
I_3&=\int_{-\theta r_j^{ps}}^{0}\int_{B_j}(|u|+|k_j|)^{q-1}{w}_j^{2}|\partial_t\psi_j(x,t)^p|\,dx\,dt\\
&\leq C\frac{2^{jps}}{\theta r^{ps}}\int_{-\theta r_j^{ps}}^{0}\int_{B_j}w_j^{q+1}\,dx\, dt\\
&\leq C\frac{2^{jps}}{\theta r^{ps}} L^{q+1}|A_j|,
\end{split}
\end{equation}
where $C=C(n,p,q,s)$.
For $q\geq 1$, we have $\mu_{-}\leq u\leq k_j\leq\mu_{-}+L$ in $A_j$. Hence, we get
\begin{equation}\label{estI3grt}
I_3=\int_{-\theta r_j^{ps}}^{0}\int_{B_j}(|u|+|k_j|)^{q-1}{w}_j^{2}|\partial_t\psi_j(x,t)^p|\,dx\,dt
\leq C\frac{2^{jps}}{\theta r^{ps}}M^{q-1}L^2|A_j|,
\end{equation}
where $C=C(n,p,q,s)$ and $M=\max\{|\mu_{-}|,|\mu_{-}+L|\}$. Therefore, we have
\begin{equation}\label{estI3final}
I_3\leq C\frac{2^{jps}}{r^{ps}}\frac{\max\{M^{q-1},L^{q-1}\}L^{2}}{\theta}|A_j|,
\end{equation}
where $C=C(n,p,q,s)$.
Since $u\leq\bar{k}_j$, we have 
\[
(1-a)2^{-j-2}L=k_j-\bar{k}_j\leq k_j-u\leq |u|+|k_j|\leq 2 M.
\] 
Therefore, for any $q>0$, we get
\begin{equation}\label{new}
(|u|+|k_j|)^{q-1}{w}_j^{2}\geq\frac{\gamma\min\{M^{q-1},L^{q-1}\}}{2^{q(j+3)}}\bar{w}_j^{2},
\end{equation}
where
\begin{equation}\label{gamma}
\gamma=\gamma(a,q)=\begin{cases}
2^{q-1},&0<q<1,\\
(1-a)^{q-1},&q\geq 1.
\end{cases}
\end{equation}
By inserting \eqref{estI1}, \eqref{estI2}, \eqref{estI3final} and \eqref{new} into \eqref{energyapp2}, we obtain
\begin{equation}\label{energyapp4}
\begin{split}
\int_{-\theta{\bar{r}_j}^{ps}}^{0}&\int_{\overline{B}_{j}}\int_{\overline{B}_{j}}\frac{|\bar{w}_j(x,t)-\bar{w}_j(y,t)|^p}{|x-y|^{n+ps}} \,dx\,dy\,dt\\
&\qquad+\frac{\gamma\min\{M^{q-1},L^{q-1}\}}{2^{q(j+3)}}\esssup_{-\theta{\bar{r}_j}^{ps}<t<0}\int_{\overline{B}_j}\bar{w}_j^2\,dx\\
&\leq C\frac{2^{j(n+ps+p)}}{r^{ps}}L^p\Big(1+\frac{\max\{M^{q-1},L^{q-1}\}L^{2-p}}{\theta}\Big)|A_j|,
\end{split}
\end{equation}
where $C=C(n,p,q,s,\Lambda)$.
By using the inequality in  Lemma \ref{parabolic embedding} with $l=p(1+\frac{2s}{n})$, we get
\begin{equation}\label{pe1}
\int_{\overline{\mathcal{Q}}_j}|\bar{w}_j|^{p(1+\frac{2s}{n})}\,dx\,dt
\leq\int_{-\theta\bar{r}_j^{ps}}^{0}\Bigl(\int_{\overline{B}_j}|\bar{w}_j|^{p\kappa^{*}}\,dx\Bigr)^\frac{1}{\kappa^{*}}\,dt
\Bigl(\esssup_{-\theta \bar{r}_j^{ps}<t<0}\int_{\overline{B}_j}|\bar{w}_j|^2\,dx\Bigr)^\frac{ps}{n}.
\end{equation}
By the Sobolev inequality in Lemma \ref{ellipticembedding2}, there exists a constant $C=C(n,p,s)$ such that
\begin{equation}\label{pe2}
\begin{split}
\Bigl(\int_{\overline{B}_j}|\bar{w}_j|^{p\kappa^{*}}\,dx\Bigr)^\frac{1}{\kappa^{*}}
\leq C\bar{r}_j^\frac{n}{\kappa^{*}}\Biggl(\bar{r}_j^{ps-n}\int_{\overline{B}_j}\int_{\overline{B}_j}\frac{|\bar{w}_j(x,t)-\bar{w}_j(y,t)|^p}{|x-y|^{n+ps}}\,dx dy+\fint_{\overline{B}_j}|\bar{w}_j|^p\,dx\Biggr).
\end{split}
\end{equation}
Noting that $\mathcal{Q}_{j+1}\subset\overline{\mathcal{Q}}_j\subset \mathcal{Q}_j$ and using \eqref{pe1} and \eqref{pe2}, we obtain
\begin{equation}\label{dist}
\begin{split}
&\frac{(1-a)L}{2^{j+2}}|A_{j+1}|=\int_{A_{j+1}}(\bar{k}_j-k_{j+1})\,dx dt
\leq\int_{\mathcal{Q}_{j+1}}\bar{w}_j\,dx dt
\leq\int_{\overline{\mathcal{Q}}_j}\bar{w}_j\,dx dt\\
&\leq\Bigl(\int_{\overline{\mathcal{Q}}_j}|\bar{w}_j|^{p(1+\frac{2s}{n})}\,dx dt\Bigr)^\frac{n}{p(n+2s)}\,|A_j|^{1-\frac{n}{p(n+2s)}}\\
&\leq C\Bigg(\bar{r}_j^\frac{n}{\kappa^{*}}\bar{r}_j^{ps-n}\int_{-\theta\bar{r}_j^{ps}}^{0}\int_{\overline{B}_j}\int_{\overline{B}_j}\frac{|\bar{w}_j(x,t)-\bar{w}_j(y,t)|^p}{|x-y|^{n+ps}}\,dx\,dy
+\int_{-\theta\bar{r}_j^{ps}}^{0}\fint_{\overline{B}_j}|\bar{w}_j|^p\,dx dt\Bigg)^\frac{n}{p(n+2s)}\\
&\qquad\cdot\Bigl(\esssup_{-\theta \bar{r}_j^{ps}<t<0}\int_{\overline{B}_j}|\bar{w}_j|^2\,dx\Bigr)^\frac{s}{n+2s}\,|A_j|^{1-\frac{n}{p(n+2s)}},
\end{split}
\end{equation}
for some constant $C=C(n,p,s)$.
By inserting \eqref{energyapp4} into \eqref{dist}, we get 
\begin{equation}\label{dist8}
\begin{split}
&\frac{(1-a)L}{2^{j+2}}|A_{j+1}|\\
&\leq C\Bigg( \bar{r}_j^\frac{n}{\kappa^{*}}\Biggl(\bar{r}_j^{ps-n}\frac{2^{j(n+ps+p)}}{r^{ps}}L^p\Big(1+\frac{\max\{M^{q-1},L^{q-1}\}L^{2-p}}{\theta}\Big)|A_j|
+\bar{r}_j^{-n}L^p |A_j|\Biggr)\Bigg)^\frac{n}{p(n+2s)}\\
&\cdot\Bigg(\frac{2^{j(n+ps+p+q)}}{r^{ps}\gamma\min\{M^{q-1},L^{q-1}\}}L^p\Big(1+\frac{\max\{M^{q-1},L^{q-1}\}L^{2-p}}{\theta}\Big)|A_j|\Bigg)^\frac{s}{n+2s}|A_j|^{1-\frac{n}{p(n+2s)}},
\end{split}
\end{equation}
for some constant $C=C(n,p,q,s,\Lambda)$.
Now since $\frac{r}{2}<\bar{r}_j<r$, we obtain from  \eqref{dist8} that 
\begin{equation}\label{dist7}
\begin{split}
|A_{j+1}|&\leq\frac{C}{(1-a)L\,r^\frac{s(n+ps)}{n+2s}}\Bigg(L^p\Big(1+\frac{\max\{M^{q-1},L^{q-1}\}L^{2-p}}{\theta}\Big)\Bigg)^\frac{n}{p(n+2s)}\\
&\qquad\cdot\Bigg(\frac{L^{p}}{\gamma\min\{M^{q-1},L^{q-1}\}}\Big(1+\frac{\max\{M^{q-1},L^{q-1}\}L^{2-p}}{\theta}\Big)\Bigg)^\frac{s}{n+2s}\\
&\qquad\cdot 2^{j\big(\frac{n(n+ps+p)}{p(n+2s)}+\frac{s(n+ps+p+q)}{n+2s}+1\big)}\,|A_j|^{1+\frac{s}{n+2s}},
\end{split}
\end{equation}
where $C=C(n,p,q,s,\Lambda)$.
By dividing both sides of \eqref{dist7} by $|\mathcal{Q}_{j+1}|$ and noting that $|\mathcal{Q}_j|<2^{ps+n}|\mathcal{Q}_{j+1}|$, and also that $r_j<r$, we obtain
\begin{equation}\label{dist4}
\begin{split}
Y_{j+1}&\leq\frac{C(n,p,q,s,\Lambda)}{(1-a)L}\Bigg(L^p\Big(1+\frac{\max\{M^{q-1},L^{q-1}\}L^{2-p}}{\theta}\Big)\Bigg)^\frac{n}{p(n+2s)}\\
&\qquad\cdot\Bigg( \frac{\theta L^{p}}{\gamma\min\{M^{q-1},L^{q-1}\}}\Big(1+\frac{\max\{M^{q-1},L^{q-1}\}L^{2-p}}{\theta}\Big)\Bigg)^\frac{s}{n+2s}\\
&\qquad\cdot 2^{j\big(\frac{n(n+ps+p)}{p(n+2s)}+\frac{s(n+ps+p+q)}{n+2s}+1\big)}\,|Y_j|^{1+\frac{s}{n+2s}},
\end{split}
\end{equation}
where we denote $Y_j=\frac{|A_{j}|}{|\mathcal{Q}_{j}|}$.
By choosing
\begin{align*}
c_0&=\frac{C(n,p,q,s,\Lambda)}{(1-a)L}\Bigg(L^p\Big(1+\frac{\max\{M^{q-1},L^{q-1}\}L^{2-p}}{\theta}\Big)\Bigg)^\frac{n}{p(n+2s)}\\
&\qquad\cdot\Bigg(\frac{\theta L^{p}}{\gamma\min\{M^{q-1},L^{q-1}\}}\Big(1+\frac{\max\{M^{q-1},L^{q-1}\}L^{2-p}}{\theta}\Big)\Bigg)^\frac{s}{n+2s},\\
b&=2^{\frac{n(n+ps+p)}{p(n+2s)}+\frac{s(n+ps+p+q)}{n+2s}+1},\,\beta=\frac{s}{n+2s},\,\tau=c_0^{-\frac{1}{\beta}}b^{-\frac{1}{\beta^2}}\in(0,1),
\end{align*}
Lemma \ref{Deg2} gives $Y_j\to 0$ as $j\to\infty$ if $Y_0\leq\tau$.
This implies that 
\[
u\geq\mu_{-}+aL
\text{ a.e. in }
\mathcal{Q}_{\frac{3r}{4}}^{-}(\theta).
\]

In the case when $ps \geq n$, we first  choose $N$ large enough such that $N > ps$ and  apply Lemma  \ref{ellipticembedding2}  with $\kappa^*= \frac{N}{N-ps}$ and Lemma \ref{parabolic embedding} with $l= p (1 + \frac{2s}{N})$ in the proof above.  Then by repeating the rest of the  arguments  above, we obtain a similar nonlinear iterative  inequality  as in \eqref{dist4} with  $\beta=\frac{s}{N+2s}$. The conclusion of the lemma thus follows in this case as well. 

\end{proof}

\medskip

%\subsection*{Proof of Theorem \ref{thm1}}
\begin{proof}[Proof of Theorem \ref{thm1}]
Let $\theta=1$. Now it is quite straightforward to see from the  proof of  Lemma \ref{Lemma1} above that $\frac{3r}{4}$ can be replaced by $ \gamma r$ for any $\gamma <1 $ in which case $\tau$ changes to a possibly smaller constant depending on $\gamma$. 
This implies that there exists  $\tau \in (0,1)$ depending on $a, L, \mu^{-}$ and the data such that 
\[
|[u \leq \mu^{-} +L] \cap \mathcal{Q}_r(x_0, t_0)| \leq \tau |\mathcal{Q}_r(x_0, t_0)|,
\]
implies $u \geq \mu^{-} + aL$ almost everywhere in $\mathcal{Q}_{cr}(x_0, t_0)$,
for some $c= c(n, p, s)\in(0,1)$.
In particular $u$ satisfies the property $(\mathcal D)$ with respect to centered cylinders $\mathcal{Q}_r(x_0, t_0)$ and  thus the desired conclusion follows  by an application of Theorem \ref{lscthm}.
\end{proof}

We state our second De Giorgi type lemma using which one can repeat the arguments in \cite{Liao} to prove Theorem \ref{thm2}. 

\begin{Lemma}\label{Lemma2}
Let $1<p<\infty,\,q,\,L>0$, $a\in(0,1)$ and $u$ be a weak supersolution of \eqref{maineqn} in $\Omega\times(0,T).$ Assume that $\mu_{-}\leq\essinf_{\mathbb{R}^n\times(0,T)}\,u$. Then there exists a constant $\theta>0$ depending only on $a,L,\mu_{-}$ and the data such that, if $t_0$ is a Lebesgue instant and
$$
u(\cdot,t_0)\geq\mu_{-}+L \text{ a.e. in }B_r(x_0),
$$
then 
$$
u\geq\mu_{-}+aL\text{ a.e. in }\mathcal{Q}^{+}\bigl(x_0,t_0;\frac{3r}{4},\theta\bigr).
$$
\end{Lemma}

\begin{proof}
We may  again assume that $ps<n$,  because  for the case $ps \geq n$ we can  then modify the arguments exactly the same way as in Lemma \ref{Lemma1}. 
Let $\theta>0$ be such that $\mathcal{Q}^{+}(x_0,t_0;r,\theta)=B_r(x_0)\times[t_0,t_0+\theta r^{ps})\Subset\Omega\times(0,T)$. 
The parameter $\theta$ will be chosen below. Without loss of generality, we may assume that  $(x_0,t_0)=(0,0)$. 
Let $k_j,\,\bar{k}_j$, $r_j$, $\bar{r}_j$, $B_j\,\overline{B}_j$, $w_j$, $\bar{w}_j$, $j=0,1,2,\dots$, be as in \eqref{parameter1} and \eqref{parameter2}. 
In contrast with \eqref{parameter3}, here we consider the forward in time cylinders 
\[
\mathcal{Q}_j=B_j\times[0,\theta r_j^{ps})
\quad\text{and}\quad
\overline{\mathcal{Q}}_j=\overline{B}_j\times[0,\theta \bar{r}_j^{ps}).
\] 
Let $\psi_j(x,t)=\psi_j(x)$ be a smooth time independent function vanishing on $\partial B_r$ such that 
$0\leq\psi_j\leq 1$ in $\mathcal{Q}_j$, $\psi_j=1$ in $\overline{\mathcal{Q}}_j$,
$$
|\nabla \psi_j|\leq C\frac{2^j}{r}
\quad\text{and}\quad
\mathrm{dist}(\mathrm{supp}\,\psi_j,\,\mathbb{R}^n\setminus B_j)\geq 2^{-j-1}r,
$$
for some constant $C=C(n,p,s)$.
By Lemma \ref{Subenergyestimate}, we obtain a constant $C=C(p,\Lambda)$ such that
\begin{equation}\label{engapp1}
\begin{split}
&\int_{0}^{\theta r_j^{ps}}\int_{B_{j}}\int_{B_{j}}|w_j(x,t)\psi_{j}(x)-w_j(y,t)\psi_{j}(y)|^p \,d\mu\,dt
+\esssup_{0<t<\theta{r_j}^{ps}}\int_{B_{j}}\psi_{j}(x)^p\xi(w_j(x,t))\,dx\\&\leq C\Bigg(\int_{0}^{\theta r_j^{ps}}\int_{B_{j}}\int_{B_{j}}{\max\{w_j(x,t),w_j(y,t)\}^p|\psi_{j}(x)-\psi_{j}(y)|^p}\,d\mu\,dt\\
&\qquad+\esssup_{x\in\mathrm{supp}\,\psi_j,\,0<t<\theta r_j^{ps}}\int_{{\mathbb{R}^n\setminus B_{j}}}{\frac{w_j(y,t)^{p-1}}{|x-y|^{n+ps}}}\,dy
\int_{0}^{\theta r_j^{ps}}\int_{B_{j}}w_j(x,t)\psi_j(x)^p\,dx\,dt\\
&\qquad+\int_{0}^{\theta r_j^{ps}}\int_{B_{j}}\xi(w_j(x,t))\partial_t\psi_j(x)^p\,dx\,dt+\int_{B_j\times\{0\}}\psi(x)^p\xi(w_j(x,t))\,dx\Bigg).
\end{split}
\end{equation}
Since $\psi$ is independent of $t$, the third term in the right hand side of \eqref{engapp1} vanishes. 
Since $u(\cdot,0)\geq\mu_{-}+L$ almost everywhere in $B_r(0)$, we have that $u(\cdot,0)\geq \mu_{-}+L>k_j$ almost everywhere in $B_j$. 
Therefore the last term on the right-hand side of \eqref{engapp1} vanishes. 
Proceeding similarly as in the proof of Lemma \ref{Lemma1},  we obtain from \eqref{engapp1} that 
\begin{equation}\label{eng1pp2}
\begin{split}
&\int_{0}^{\theta{\bar{r}_j}^{ps}}\int_{\overline{B}_{j}}\int_{\overline{B}_{j}}\frac{|\bar{w}_j(x,t)-\bar{w}_j(y,t)|^p}{|x-y|^{n+ps}} \,dx\,dy\,dt+\frac{\gamma\min\{M^{q-1},L^{q-1}\}}{2^{jq}}\esssup_{0<t<\theta{\bar{r}_j}^{ps}}\int_{\overline{B}_j}\bar{w}_j^2\,dx\\
&\qquad\leq C\frac{2^{j(n+ps+p)}}{r^{ps}}L^p|A_j|,
\end{split}
\end{equation}
where $C=C(n,p,q,s,\Lambda)$, $A_j=\{u<k_j\}\cap \mathcal{Q}_j$, $\gamma$ as defined in \eqref{gamma} and $M=\max\{|\mu_{-}|,|\mu_{-}+L|\}$. 
Observing that $\mathcal{Q}_{j+1}\subset\overline{\mathcal{Q}}_j\subset \mathcal{Q}_j$ and, by arguing similarly as in  Lemma \ref{Lemma1}, we conclude that
\begin{equation}\label{iterationest}
\begin{split}
Y_{j+1}&\leq\,\frac{C}{1-a}\Big(\frac{\theta}{L^{2-p}\gamma\min\{M^{q-1},L^{q-1}\}}\Big)^\frac{s}{n+2s}\,2^{j\bigl(\frac{n(n+ps+p)}{p(n+2s)}+\frac{s(n+ps+p+q)}{n+2s}+1\bigr)}\,Y_j^{1+\frac{s}{n+2s}},
\end{split}
\end{equation}
where $C=C(n,p,q,s,\Lambda)$ and $Y_j=\frac{|A_{j}|}{|\mathcal{Q}_{j}|}$. Let
\[
d_0=\frac{C}{1-a}\Big(\frac{1}{L^{2-p}\gamma\min\{M^{q-1},L^{q-1}\}}\Big)^\frac{s}{n+2s},
\]
where $C=C(n,p,q,s,\Lambda)$.
By choosing
\[
b=2^{\bigl(\frac{n(n+ps+p)}{p(n+2s)}+\frac{s(n+ps+p+q)}{n+2s}+1\bigr)},
\quad
\beta=\frac{s}{n+2s},
\quad
c_0=d_0\,\theta^\beta
\]
in Lemma \ref{Deg2}, we have $Y_j\to0$ as $j\to\infty$, if
$Y_0\leq\nu=c_0^{-\frac{1}{\beta}}b^{-\frac{1}{\beta^2}}$. 
By choosing $\theta=\delta\,d_0^{-\frac{1}{\beta}}\,b^{-\frac{1}{\beta^2}}$ for $\delta\in(0,1)$, we get $\nu=\delta^{-1}>1$. 
Hence the fact  that $Y_0\leq 1$ and Lemma \ref{Deg2} imply that
$Y_j\to0$ as $j\to\infty$. 
Therefore, we have 
\[
u\geq\mu_{-}+aL
\text{ a.e. in }
\mathcal{Q}_{\frac{3r}{4}}^{+}(\theta).
\]
\end{proof}

\section*{Acknowledgements}
{A.B. is supported in part by SERB Matrix grant MTR/2018/000267 and by Department of Atomic Energy,  Government of India, under
project no. 12-R \& D-TFR-5.01-0520.
P.G. and J.K. are supported by the Academy of Finland.}

\medskip
\noindent
Agnid Banerjee\\
Tata Institute of Fundamental Research\\
Centre For Applicable Mathematics\\
Bangalore-560065, India\\
Email: agnidban@gmail.com\\

\noindent
Prashanta Garain\\
Department of Mathematics\\
Ben Gurion University of the Negev\\
Be’er Sheva-84105, Israel\\
Email: pgarain92@gmail.com\\

\noindent
Juha Kinnunen\\
Department of Mathematics\\
P.O. Box-11100\\
FI-00076 Aalto University, Finland\\
Email: juha.k.kinnunen@aalto.fi

\end{document}